\documentclass[11pt]{amsart}

\usepackage{amssymb, graphicx}
\usepackage{graphicx}

\usepackage{amsfonts}
\usepackage{amsmath}
\usepackage[foot]{amsaddr}
\usepackage{latexsym}
\usepackage{amscd}
\usepackage{verbatim}
\usepackage{color}

\newcommand{\mf}{\mathfrak}
\newcommand{\g}{\mf{g}}
\newcommand{\h}{\mf{h}}

\renewcommand{\l}{\lambda}

\allowdisplaybreaks[1]

\numberwithin{equation}{section}
\newtheorem{theorem}{Theorem}[section]
\newtheorem{proposition}[theorem]{Proposition}
\newtheorem{lemma}[theorem]{Lemma}
\newtheorem{corollary}[theorem]{Corollary}
 
\theoremstyle{remark}
\newtheorem{remark}[theorem]{Remark}
\newtheorem{definition}[theorem]{Definition}

\renewcommand{\a}{\alpha}

\newcommand{\C}{\mathbb{C}}
\newcommand{\D}{\Delta}
\newcommand{\Dp}{\Delta^+}

\begin{document}

\title{On some modules of covariants for a reflection group}

\date{\today}
\dedicatory {To Ernest Vinberg on the occasion of his 80th birthday}

\author{Corrado De Concini,
Paolo Papi}\address{Dipartimento di Matematica, Sapienza Universit\`a di Roma, P.le A. Moro 2,
00185, Roma, Italy; 
}
\email{deconcin@mat.uniroma1.it, papi@mat.uniroma1.it }
  \keywords{Exterior algebra, covariants,   small representation, Dunkl operators.}
  \subjclass[2010]{17B20}

\date{\today}
\begin{abstract} Let $\mathfrak g$ be a simple Lie algebra with  Cartan subalgebra $\h$ and Weyl group $W$. We build up a graded isomorphism   $(\bigwedge\h\otimes\mathcal H\otimes \h)^W\to (\bigwedge \g\otimes \g)^\g$ of 
 $(\bigwedge \g)^\g\cong S(\h)^W$-modules, where  $\mathcal H$ is the space of $W$-harmonics.
In this way we prove an enhanced form of a conjecture of Reeder for the adjoint representation.
\end{abstract}

\maketitle

\section{Introduction}
Let $W$ be a  finite irreducible real reflection group, and let $S$ be a set of Coxeter generators. Let $V$ be the  euclidean space  affording a reflection representation of $W$.
Consider the ring $A$ of complex valued polynomial functions on $V$. Let $2\leq d_1\leq d_2\leq\cdots\leq d_r,\,r=\dim V$ be the degrees of any set  of homogeneous generators $\psi_1,\ldots ,\psi_r$ of the polynomial ring $A^W$. Now consider the ideal $J$ of $A$ generated by  $\psi_1,\ldots ,\psi_r$  and set $\mathcal H=A/J$.  \par
Let  $\mathcal W=\bigwedge V\otimes A$ be the Weil algebra, which we regard as graded by  
$\deg(q\otimes k)=\deg(q)+2\deg(k)$ for $q\in \bigwedge V$ and $k\in A$ homogeneous elements.  Consider now the graded ring $\mathcal B=\bigwedge V\otimes \mathcal H=\oplus_q\mathcal B_q$ and its special elements
\begin{equation}\label{elementi}
p_i=\pi(d (1\otimes\psi_i))\in \mathcal B^W,
\end{equation}
$d$ being the De Rham differential on $\mathcal W$  (cf. \eqref{delta}) and $\pi: \mathcal W\to \mathcal B$ the quotient map. A classical  theorem of Solomon states that $\mathcal B^W=\bigwedge(p_1,\ldots,p_r)$ (cf. Proposition \ref{Solo}).
Let $\mathcal D=\hom_W(V,\mathcal B)$. Fix a $W$-invariant non degenerate bilinear $(-,-)$ form on $V$   (unique up to multiplication by a non zero constant). There is a natural $\mathcal W$-valued bilinear form $E$ on $\mathcal W \otimes V$ defined by 
\begin{equation}\label{formaE}E(w_1\otimes v_1, w_2\otimes v_2)= (v_1,v_2) w_1w_2,\end{equation}
for $v_1,v_2\in V,\,w_1,w_2\in \mathcal W$. Since $J$ is an ideal, the form pushes down to $\mathcal B \otimes V\cong \hom(V,\mathcal B)$, where we identify $V$ with $V^*$ using the bilinear form $(-,-)$. Passing to the invariants, we obtain 
 a $\mathcal B^W$-valued bilinear form, still denoted by $E$,  on the $\mathcal B^W$-module
$\mathcal D=\hom_W(V,\mathcal B)$.

Our main result is the following theorem,  a more precise version of which  is given in Theorem \ref{main} (see also Proposition \ref{prop}). \begin{theorem}\label{main0}
{\bf (1).}
  $\mathcal D$ is a free module, with explicit generators $f_i,u_i, i=1,\ldots,r,$ over the exterior algebra $\bigwedge(p_1,\ldots,p_{r-1}).$\par\noindent
{\bf (2).} There are non zero constants  $k_i\in \mathbb Q$ such that $$E(f_i,u_{r-i+1})=k_ip_r$$ for each $i=1,\ldots,r$.  The multiplication by $p_r$ is self adjoint for the form $E$. It is  given by the formulas
\begin{align}
p_rf_i= -\sum_{j=1,\ j\neq i}^{r} k_j^{-1}E(f_i,u_{r-j+1})f_j,\qquad  i=1,\ldots,r,\\
p_ru_i= -\sum_{j=1,\ j\neq i}^{r} k_j^{-1} E(f_i,u_{r-j+1})u_j ,\qquad i=1,\ldots,r,
\end{align} 
\end{theorem}
Statement (1) has been proven, for well--generated complex reflection groups, in \cite{RS}. 
\par 
As a consequence  of Theorem \ref{main0}Ê  we give a positive answer to a special case of a conjecture of Reeder, in an
``enhanced''  formulation due to Reiner and Shepler: see Section \ref{dueprimo} for details. \par
{\bf Acknowledgement.} The authors wish to thank the referee whose suggestion greatly improved the presentation of the paper. 
\section{Preliminaries, Motivations and outline of proof of Theorem \ref{main0}}\label{dueprimo}
The framework of Reeder's conjecture is Lie-theoretic,  so let us revert to this context and fix notation.\par
Let $\mathfrak g$ be a finite-dimensional simple Lie algebra (over $\mathbb C$) of rank $r$. Fix a Cartan subalgebra $\h$ in $\g$. Let $\Delta$ be the corresponding root system, $W$ the Weyl group, $\Delta^+$ a positive system and $\rho$ the Weyl vector. Observe that as a $W-$module, $\h$ is the reflection representation. We will identify $\g$ and $\g^*$ via the Killing  form which restricts to a $W-$invariant bilinear form on $\h$ which we choose as our form $(-,-)$.  Let $Q, P$ denote the root and weight lattices, $P^+$ the cone of dominant integral weights, 

The exterior algebra $\bigwedge \mathfrak g$ has been extensively studied as representation of $\mathfrak g$ (see e.g. \cite{Kold}, \cite{K}). 
We are concerned with Reeder's paper  \cite{R}, where the author studies the isotopic components  in $\bigwedge\g$ of representations whose highest weight is ``near''  $2\rho$ or ``near''  $0$ w.r.t. the usual partial order on dominant weights. The nearness condition about $0$ is made precise in the following 
\begin{definition} A irreducible finite dimensional highest weight module $V_\l$  with highest weight $\l\in Q\cap P^+$ is said to be {\it small} if twice a root of $\g$ is not a weight of $V_\l$.
\end{definition}

Given $\l\in Q\cap P^+$, the zero-weight space $0\ne V_\l^0\subset V_\l$ is a $W$-module. Introduce the following 
generating functions:
$$P(V_\l,\bigwedge \g, u)=\sum_{n\geq 0}\dim \hom_\g(V_\l,\bigwedge^n \g)u^n,\quad
P_W(V^0_\l,\mathcal B,u)=\sum_{q\geq 0}\dim \hom_W(V^0_\l,\mathcal B_q)u^q.
$$
In \cite[Conjecture 7.1]{R} Reeder proposed the following relation between these generating series when $V_\l$ is small:
\begin{equation}\label{RCC}
P(V_\l,\bigwedge \g, u)=P_W(V^0_\l, \mathcal B, u),
\end{equation}
and verified it in rank less or equal then $3$.  
 The conjecture has two different motivations. Let $G$ be a compact Lie group with complexified Lie algebra $\g$ and let $T\subset G$ be a maximal torus. Consider the $W$-action on both factors of the manifold $T\times G/T$. The Weyl map $T\times_W G/T\to G$ induces  an isomorphism in cohomology, which in terms of invariants reads as 
an isomorphism of graded vector spaces 
$$(\bigwedge\g)^\g\cong H^*(G)\cong H^*(T\times G/T)^W\cong (\bigwedge \h^*\otimes\mathcal H )^W=\mathcal B^W.$$ Conjecture \eqref{RCC} is the natural extension of this graded isomorphism to covariants of small representations.
On the other hand, Broer \cite{Broer} has shown that,  exactly for small representations, Chevalley restriction can be generalized to covariants. Let  $S(\g)$ (resp. $S(\h )$) denote   the symmetric algebra of $\g$ (resp. of $\h$). Chevalley restriction theorem gives an isomorphism $S(\g)^\g\simeq S(\h)^W$. Broer proves that restriction also induces an   isomorphism of graded $S(\g)^\g\simeq S(\h)^W-$modules between $\hom_\g(V_\l,S(\g))$ and $\hom_W(V^0_\l,S(\h))$.\par
Curiously enough, conjecture \eqref{RCC} in type $A$ was implicitly proven in literature before \cite{R}  appeared: the left hand side has been computed by Stembridge \cite{S}, whereas the right hand side appears in \cite{KP}, \cite{M} (in a more general context). Further related work appears in \cite{SS}, where Stembridge provides methods which  can be reasonably applied for a case by case proof of the conjecture (see the discussion at the end of Section 3 in \cite{RS}).

Set $\Gamma=(\bigwedge \g)^\g\cong \mathcal B^W$. In Corollary \ref{CC} we prove that Theorem \ref{main0} implies the following 
\begin{theorem}\label{RC}  There is a degree preserving isomorphism of $\Gamma$-modules
\begin{equation}\label{iso} ( \bigwedge\h\otimes \mathcal H\otimes\h)^W\to (\bigwedge \g\otimes \g)^\g.\end{equation}\end{theorem}
We are also able to build up a module isomorphism like \eqref{iso} for the little adjoint representation, i.e., the highest weight module $\g_s$ with highest weight the highest short root of $\D$ (provided two different root lengths exist): see  Corollary \ref{C}. Indeed, in Section \ref{la}, we prove an analogue of Theorem \ref{main0} for  the ``Weyl group side'' of the little adjoint representation: see Theorem \ref{main45}.\par
The  statement of Theorem \ref{RC}Ê cannot be extended from  the adjoint representation to a general  small representation: the small module $S^3(\C^3)$ for $\g=sl(3,\C)$ admits as zero-weight  space
the sign representation $sign$ of the symmetric group $S_3$, but an easy analysis shows that a graded isomorphism of $\Gamma$-modules $$\hom_{S_3}(sign, \bigwedge\h\otimes \mathcal H)\cong 
\hom_{sl(3,\C)}(S^3(\C^3),\bigwedge sl(3,\C))$$ cannot exist. Nevertheless in Section \ref{FR} Êwe provide a speculative approach to a possible extension of Reeder's conjecture. We  build up, for any $\g$-module $V$,  a map $\Phi^V$ from covariants of type $V$  in $\bigwedge \g$ to 
covariants of type $V^0$ in $\bigwedge\h\otimes \mathcal H$ (see \eqref{Phiii}). We conjecture that  $\Phi^V$ is injective for any $V$. A result of Reeder   would then imply that $\Phi^V$ is an {\it isomorphism  of graded vector spaces} when $V$ is small, hence implying Reeder's conjecture.\par

Our approach to Theorem \ref{main0} is motivated by our previous work with Procesi on covariants of the adjoint representation in $\bigwedge \g$ \cite{DPP}. 
It is a classical fact that 
the invariant algebra  $\Gamma$  is  an exterior algebra  $\bigwedge(P_1,\ldots,P_r)$ over primitive generators $P_i$ of degree  $2d_i-1$. The main subject  of \cite{DPP} is the study of the module of covariants $\mathcal A=\hom_\g(\g,\bigwedge\g)$; we prove the following three facts (assume for simplicity of exposition that all exponents $1=m_1\leq\ldots\leq m_r $ of $\g$ are distinct).\par
{\bf (1).} $\mathcal A$ is a free module over $\bigwedge(P_1,\ldots,P_{r-1})$ of rank $2r$. A set of free generators is given by    the $\g$-equivariant maps 
\begin{equation*}
\label{gliuf} f^\wedge_i(x)=\frac{1}{\deg(P_i)} \iota (x)P_i,\quad u^\wedge_i(x)=\frac{2}{\deg(P_i)} \iota(\mathbf{d}(x))P_i,\quad i=1,\ldots,r
\end{equation*}
where $x\in\g$, $\iota$ denotes interior multiplication in the exterior algebra and $\mathbf{d}$ is the usual Chevalley-Eilenberg coboundary operator for Lie algebra cohomology.
\par
{\bf (2).}
The  Killing form on $\g$ induces an invariant  graded symmetric bilinear $\bigwedge  \mathfrak g $-valued form on $\bigwedge  \mathfrak g\otimes\mathfrak g $
given, for $a,b\in\bigwedge \g,\,x,y\in \g$, by
\begin{equation*}\label{formae}e(a\otimes x,b\otimes y)=(x,y)a\wedge b,\end{equation*}
which restricts to a $\Gamma$-valued form on  $\mathcal A$. Then, for each pair $i,j$ there exists a  non-zero rational constant  $c_{i,j}$ such that
\begin{align}\label{ps}
&e(f^\wedge_i,f^\wedge_j)=e(u^\wedge_i,u^\wedge_j)=0.\\
&\label{cost} e(f^\wedge_i,u^\wedge_j)=e(f^\wedge_j,u^\wedge_i)=\begin{cases}
c_{i,j}P_k\quad&\text{if}\quad m_i+m_j-1=m_k\quad \text{is an exponent,}\\
0\quad &\text{otherwise.}
\end{cases}
\end{align}
\par
{\bf (3).}\label{lecostanti} Set $c_i:=c_{i,r-i+1}$. The $\Gamma$-module structure of $\mathcal A$ is expressed by the following relations 
\begin{align}\label{p1}
P_r f_i^\wedge&= -\sum_{j=1,\ j\neq i}^{r} c_j^{-1}e(f^\wedge_i,u^\wedge_{r-j+1})f^\wedge_j,\qquad  i=1,\ldots,r,\\\label{p2}
P_r  u_i^\wedge&= -\sum_{j=1,\ j\neq i}^{r} c_j^{-1} e(f^\wedge_i,u^\wedge_{r-j+1})u_j ,\qquad i=1,\ldots,r.
\end{align} 
Similar results are obtained in \cite{DPP2} Êfor covariants of the little adjoint representation.\par
In  Section \ref{due} we define, in the context of finite reflection groups, equivariant maps $u_i, f_i\in \hom_W(\h,\mathcal H\otimes \bigwedge\h), i=1,\ldots,r$ of suitable degrees for which statements {\bf (1), (2), (3)} hold upon replacing 
$P_i,e,  u^\wedge_i, f^\wedge_i$ with $p_i, E, u_i, f_i$, respectively. 
\par 
The definition \eqref{f} of the $f_i$ is natural after definition \eqref{elementi}. 
The key technical point is getting  the analog of relations \eqref{ps}, \eqref{cost}. For that purpose  it is necessary to introduce  carefully chosen elements $u_i$, whose definition  \eqref{u} involves a variation  of Dunkl's operators. We are then able to prove Proposition \ref{prop} in the adjoint setting and Proposition \ref{prop23} in the little adjoint setting, which are  the ``symmetric'' analogs of statement  (2).

\section{Symmetric picture}\label{due}
As in the Introduction, let $W$ be a  finite irreducible real reflection group, and let $S$ be a set of Coxeter generators. Let $V$ be the  euclidean space  affording a reflection representation of $W$ (which we assume to be irreducible) and 
$(\cdot , \cdot)$ the  positive definite  $W$-invariant symmetric bilinear on $V$. We will identify $V$ and $V^*$ when convenient via the invariant bilinear form. Let $T\subset W$ be the set of reflections. It is well-known  that $T$ is the union of at most two conjugacy classes $T_\ell$ and $T_p$. 
Let choose for every $s\in T$ a non zero vector $\alpha_s$ orthogonal to the reflection hyperplane $Fix(s)$ (so that $s(\alpha_s)=-\alpha_s)$, and let $\Delta^+$ be the set of such vectors; 
then $\Delta^+\cup -\Delta^+$ is a root system in the sense of \cite[1.2]{Hum}.\par
 Consider the ring $A$ of complex valued polynomial functions on $V$ (which is also $S(V)$, under the identification $V\cong V^*$). One knows that $A^W$ is a polynomial ring on dim$V=r$ homogeneous generators $\psi_1,\ldots ,\psi_r$ of degrees $2\leq d_1\leq d_2\leq\cdots\leq d_r$. Now consider the ideal $J$ of $A$ generated by 
$\psi_1,\ldots ,\psi_r$  and set $\mathcal H=A/J$. This is a graded representation of $W$ whose ungraded character is the regular character.  It is a well known fact that $A\simeq A^W\otimes \mathcal H$ as a $A^W$-module.  
\par
Let  $\mathcal W=\bigwedge V\otimes A$ be the Weil algebra, which we regard as graded by  
$\deg(q\otimes k)=\deg(q)+2\deg(k)$ for $q\in \bigwedge V$ and $k\in A$ homogeneous elements.  Using the duality between $V$ and $V^*$ we think of $\mathcal W$  as the algebra of differential forms on $V$ with polynomial coefficients.

So, $\mathcal W$ is equipped with the usual de Rham differential
 $d$ given by  
\begin{equation}\label{delta}d(q\otimes k)=\sum_{i=1}^r(x_i\wedge q)\otimes \frac{\partial k}{\partial x_i},\end{equation}
where $\{x_1,\ldots x_r\}$ is an orthonormal basis of $V$. 
Under our grading, $d$ has clearly  degree $-1$.

Consider now the graded ring $$\mathcal B=\bigwedge V\otimes \mathcal H=\mathcal W /\mathcal W J.$$ We denote by $\pi:\mathcal W \to \mathcal B$ the quotient homomorphism; sometimes, abusing notation, we also denote by $\pi$ the quotient map $A\to \mathcal H$. It is clear that $\mathcal B$ inherits a grading from $\mathcal W$.

Together with $d$ we also have the Koszul differential $\delta$   given by the derivation  
\begin{equation}\label{ruledelta}
\delta(x_i\otimes 1)=1\otimes x_i,\qquad\delta(1\otimes f)=0, f\in A,
\end{equation}
which has degree $1$ under our grading.

Since $\delta$ is $W$-equivariant and  the ideal $\mathcal W J$ is preserved by $\delta$,   $\delta$ induces a differential on $\mathcal B$. 
On the other hand, $\mathcal W J$ is clearly not preserved by $d$ and we need to introduce a further differential on $\mathcal W$.



 For $s\in T$ consider the operator
$$\nabla_{s}=(d\log\a)(1-s)=(\a\otimes 1) \frac {1-s}{\a}, $$ 
where $\a=\a_s\in \Dp$. Remark that $\nabla_{s}$ does not depend on the choice of $\a$ and acts on the Weil algebra $\mathcal W$.
The following properties of $\nabla_s$ are clear from its definition.
\begin{lemma}\label{l21}\begin{enumerate}\item If $\omega\in \mathcal W^W$, $\nabla_{s}(\omega)=0.$
\item $\nabla_{s}(\omega\nu)=\nabla_{s}(\omega)\nu+(s\omega)\nabla_{s}(\nu),$ $\omega,\nu\in \mathcal W$.
\end{enumerate}\end{lemma}
Lemma \ref{l21} implies  that the ideal $\mathcal W J$ is preserved by $\nabla_{s}$, so we get an operator on the algebra $\mathcal B$. 

 We  now remark that if $\omega=a\otimes b$, $a\in\bigwedge V,\ b\in A$,
$a\otimes b-s(a\otimes b)=(a-s(a))\otimes b+s(a)\otimes (b-s(b))$ and we have
\begin{lemma} If $a\in \bigwedge V$, $\alpha_s\wedge (a-s(a))=0$.\end{lemma}
\begin{proof} If $x\in V$, $x-s(x)$ is a multiple of $\alpha_s$ and we are done. Let $a=a'\wedge x'$ with $a'$ of degree $t$ and $x\in V$. Then, by induction,  $\alpha_s\wedge (a-s(a))=\alpha_s\wedge ((a'-s(a')\wedge x'+s(a')\wedge (x'-s(x'))=\alpha_s\wedge(a'-s(a'))\wedge x'+(-1)^ts(a')\wedge \alpha_s\wedge (x'-s(x'))=0$.\end{proof}

It follows that 
\begin{equation}\label{psi1}\nabla_s(a\otimes b)=(\alpha_s\wedge s(a))\otimes \frac{ (1-s)(b)}{\alpha_s}.\end{equation}

We now choose a function $c:T\to \mathbb C$ constant on conjugacy classes and set 
\begin{equation}\label{partial}D_c:=\sum_{s\in T}c(s)\nabla_s,\end{equation}
and consider it as an operator both on $\mathcal W$ and on $\mathcal  B$.
Notice that, since clearly $D_c(\mathcal W^W)=0$ and any element of $\mathcal B^W$ can be lifted to a $W$-invariant element of $\mathcal W$, we get
\begin{proposition}\label{primo} If $u\in \mathcal B^W$, $D_c(u)=0$.
\end{proposition}

\begin {lemma} If $w\in W$, 
$$w^{-1}D_c w=D_c.$$
\end{lemma}
\begin{proof} We have $$w^{-1}D_c w(\omega)=w^{-1}( \sum_{s\in T}c(s)\nabla_s(w\omega))=\sum_{{s\in T}}c(s)\nabla_{w^{-1}sw}(\omega)=D_c (\omega),$$
since the function $c$ is constant on conjugacy classes.
\end{proof}

\begin{proposition}\label{ddelta} Let $U$ be an irreducible $W$-module and $x\in \mathcal H$ or $x\in A$ be such that it generates a copy of $U$. Fix $s_\ell\in T_\ell$, $s_p\in T_p$. Then 
\begin{equation}\label{formuladelta}\delta D_c(x)=(c(s_\ell)(|T_\ell|(1-\frac{\chi_U(s_\ell)}{\chi_U(1)})+c(s_p)|T_p|(1-\frac{\chi_U(s_{p})}{\chi_U(1)}))x.\end{equation} 
\end{proposition}
\begin{proof} By the definitions
$$\delta D_c(x)=\sum_{s\in T} c(s)(x-s(x)),$$
so that $\delta D_c(x)\in U$. Since $U$  is irreducible and $\delta D_c$ commutes with the $W$-action, we get that $\delta D_c(x)=\gamma x$, $\gamma$ a constant. Computing traces we get
$$\gamma \chi_U(1)=(c(s_\ell)|T_\ell |+c(s_p)|T_p |)\chi_{U}(1)-c(s_\ell)|T_\ell |\chi_U(s_\ell)-c(s_p)|T_p|\chi_U(s_p),$$ from which \eqref{formuladelta} is clear.\end{proof}

Finally we see that $D_c$ gives a differential both on $\mathcal W$ and on $\mathcal B$. Indeed we have
\begin{proposition}\label{secondo'} 
$D_c^2=0$.
\end{proposition}
\begin{proof}  
We have
$$D_c^2=\sum_{(s,t)\in T\times T}c(s)c(t)\nabla_s\nabla_t.$$
Now
$$\nabla_s\nabla_t(a\otimes b)=(\alpha_s\wedge s(\alpha_t)\wedge st(a))\otimes (\frac{b-t(b)}{\alpha_s\alpha_t}-\frac{s(b)-st(b)}{ \alpha_s s(\alpha_t)})$$
If $s=t$, clearly $\alpha_s\wedge s(\alpha_s)=-\alpha_s\wedge \alpha_s=0$, so we can assume $s\neq t$.

We now consider the space $V_{s,t}$ spanned by $\alpha_s$ and $\alpha_t$ and the dihedral subgroup $W_{s,t}$ generated by $s,t$. If we set $U=\a_t^\perp \cap \a_s^\perp$, we clearly get that $V=V_{s,t}\oplus U$ and we can write as a linear combination of elements of the form $a=a'\otimes u$ with $a'\in\bigwedge V_{s,t}$ and $u\in \bigwedge U$ each homogeneous.
Then if $a'$ is of positive degree we get $\alpha_s\wedge s(\alpha_s)a'=0$ so that we get possibly non zero contributions to $\nabla_s\nabla_t(a\otimes b)$ only when $a'=1$. By linearity we can assume that $a\in \bigwedge U$  so that $st(a)=a$ and we get 
$$\nabla_s\nabla_t(a\otimes b)=(a\otimes 1)(\alpha_s\wedge s(\alpha_t)\otimes (\frac{b-t(b)}{\alpha_s\alpha_t}-\frac{s(b)-st(b)}{ \alpha_s s(\alpha_t)})).$$
We can even assume that $a=1$ and look at
$$\nabla_s\nabla_t(1\otimes b)=(\alpha_s\wedge s(\alpha_t))\otimes (\frac{b-t(b)}{\alpha_s\alpha_t}-\frac{s(b)-st(b)}{ \alpha_s s(\alpha_t)}).$$
Furthermore notice that all the contributions to the right hand side come from either multiplying of dividing by vectors in $V_{s,t}$  or applying elements in $W_{s,t}$. From these considerations we deduce that we can really assume that $W=W_{s,t}$ and the claim follows from  Lemma \ref{lemma} below.\end{proof}
Consider a dihedral group $D$ generated by the reflections $s,t$ subject to the relation $(st)^n=1$, so that its set of reflections is formed by the $n$ elements $s_1=s,s_2= sts, s_3= ststs, \ldots s_n=(st)^{n-1}s=t$.

Let $\mathfrak h$ be a reflection representation of $D$ (meaning that is the direct sum  $\mathfrak h_{s,t}\oplus U$, with $\h_{s,t}$ the irreducible 2 dimensional reflection representation of $D$  as above). We choose as usual $\alpha_i=\alpha_{s_i}$, $i=1, \ldots , n$ and consider the ring $R=\bigwedge \h_{s,t}\otimes S(\mathfrak h)[\prod \alpha_i^{-1}]$ 
 and the twisted group algebra $R[D]$.
\begin{lemma}\label{lemma} The element $$\sum_{r=1}^n c(s_r)d\log\alpha_r(1-s_r)$$
has zero square.
\end{lemma}
\begin{proof} As we have seen, each summand of $$\left(\sum_{i=1}^nc(s_r) \alpha_r\otimes \frac{1}{\alpha_r}(1-s_r)\right)^2$$ comes from a pair of reflections $(s_i, s_j)$ and is of the form 

$$q_iq_j(\alpha_i\wedge s_i(\alpha_j))\otimes \left(\frac{1-s_j}{\alpha_i\alpha_j}-\frac{s_i(1-s_j)}{ \alpha_is_i(\alpha_j)}\right),$$
where we set $c(s_h)=q_h$ for each $h$.
So the pairs $(s_i,s_j)$ give to $s_j$ the contribution
\begin{equation}\label{prima}q_j\sum_{i=1}^n q_i(\alpha_i\wedge s_i(\alpha_j))\otimes \frac{1}{\alpha_i\alpha_j}.\end{equation}
On the other hand the pairs $(s_j,s_i)$ give to $s_j$ the contribution
\begin{equation}\label{seconda}q_j\sum_{i=1}^n q_i(\alpha_j\wedge s_j(\alpha_i))\otimes \frac{1}{\alpha_js_j(\alpha_i)}.\end{equation}
We have already seen that we can assume that $j\neq i$. Then setting $\alpha_h=s_j(\alpha_i)$, and observing that $q_h=q_i$, we get that \eqref{seconda} becomes 
$$q_j\sum_{h=1}^n q_h(\alpha_j\wedge \alpha_h)\otimes \frac{1}{\alpha_j\alpha_h}.$$
On the other hand, $\alpha_i\wedge s_i(\alpha_j)=\alpha_i\wedge \alpha_j$, so that \eqref{prima} equals $$q_j\sum_{i=1}^n q_i(\alpha_i\wedge \alpha_j)\otimes \frac{1}{\alpha_i\alpha_j}.$$
Thus the coefficient of $s_j$ is clearly $0$.

We now pass to the coefficient of $s_is_j$.  
This is equal to 
$$q_iq_j(\alpha_i\wedge s_i(\alpha_j))\otimes \frac{1}{ \alpha_is_i(\alpha_j)}=q_iq_j(\alpha_i\wedge \alpha_j)\otimes \frac{1}{ \alpha_is_i(\alpha_j)}$$
For each $h=1,\ldots n$, $s_is_j=s_hs_{h+j-i}$ and $q_iq_j=q_hq_{h+j-i}$.
So one needs to verify
$$\sum_{h=1}^nd\log\alpha_h\wedge d\log s_h(\alpha_{h+j-i})=\sum_{h=1}^nd\log\alpha_h\wedge d\log\alpha_{h+i-j}=0.$$
If we take a cycle $c=(u_1,u_2,\ldots ,u_d)$, we claim  that, setting  $u_{d+1}=u_1$, 
$$\sum_{r=1}^{d}d\log\alpha_{u_i}\wedge d\log(\alpha_{u_{i+1}})=0.$$
If the cycle has length $2$ this is obvious. If  it has length $3$, a simple computation shows that $$d\log\alpha_{u_1}\wedge d\log\alpha_{u_2}+d\log\alpha_{u_2}\wedge d\log\alpha_{u_3}
=d\log\alpha_{u_1}\wedge d\log\alpha_{u_3}$$
which is our relation.

Proceed now  by induction and,  using above relation,
substitute and get the relation using the cycle $(u_1,u_3,\ldots u_d)$.
Let us fix $m=j-i$ and consider the permutation $\sigma(h)=m+h, mod(n)$ (choosing as remainders $1,\ldots ,n$).
 Decompose it into cycles and apply the previous claim to get the result.\end{proof}

 \begin{remark} We are going to call $D_c$ a Dunkl differential. Operators of this kind on differential forms already appear in the paper \cite{Dun}.\end{remark}

\section{The bilinear form}

If $W$ is crystallographic, hence it is the Weyl group associated to a simple Lie algebra $\mathfrak g$, we recall that, by Chevalley theorem, restriction gives an isomorphism between $S(\mathfrak g)^{\mathfrak g}$  and $A^W$, the polynomial ring  of $W$ invariant functions on the Cartan subalgebra. We then fix 
 homogenous generators $\psi_1,\ldots, \psi_r$ of the polynomial ring $\mathbb C[\mathfrak g]^{\mathfrak g}\simeq A^W$ in such a way that they induce by transgression the generators $P_1,\ldots P_r$  of $(\bigwedge \mathfrak g)^{\mathfrak g}$ considered in the Introduction. On the other hand considering   $\psi_1,\ldots, \psi_r$ in $A^W$, we can introduce the elements 
$p_i$ (cf. \eqref{elementi}).

In the case $W$ is not crystallographic, we choose the homogenous generators $\psi_1,\ldots, \psi_r$ of the polynomial ring $A^W$ arbitrarily and proceed to define the elements 
$p_i$, $i=1,\ldots ,r$ as above.

\begin{remark} A priori the definition of the elements $p_i$ depends on the choice of the  generators $\psi_1,\ldots, \psi_r$ of the polynomial ring $A^W$. However, if  $J\subset A^W$ denotes as above  the ideal of elements of positive degree it is immediate to see that  the $p_i$ depend only on the induced basis  $\overline \psi_1,\ldots  \overline \psi_r$ of $J/J^2$.
 
 Indeed if $z=\psi_i-\psi'_i\in J^2$ and $z=\sum_jx_jy_j, x_i,y_j\in J$, then 
$$\pi(d (1\otimes z))=\pi(\sum_jd(x_{j})y_{j}+\sum_jx_{j}d(y_{j}))=0,$$
 proving the claim.
\end{remark}

We have the following theorem of Solomon,  which is reproved here for the reader's convenience.
\begin{proposition}\label{Solo} $\mathcal B^W$ is the graded exterior algebra generated by the elements
$p_i$ defined in \eqref{elementi}.
\end{proposition}
\begin{proof} First notice that  dim $\mathcal B^W$=dim $\bigwedge V=2^r$. Secondly, notice that for each $i=1,\ldots ,r$ the element $p_i$  is of degree $(1,2d_i-2)$, that is of total degree $2d_i-1$. It is clear that $p_ip_j=-p_jp_i$, so it suffices to show that $\prod_{j=1}^rp_j\neq 0$.
Now let us remark that the element
$$\Delta=\det (\frac{\partial_i\psi_j}{\partial x_i})$$
spans the copy of the sign representation of $W$ of lowest possible degree. Thus $\Delta\notin J$. Furthermore 
$$\prod_{j=1}^rp_j=x_1\wedge x_2\cdots \wedge x_r\otimes \pi(\Delta)\neq 0$$
and the claim follows.
\end{proof}
\begin{remark} In the crystallographic case, we get a natural isomorphism between $(\bigwedge \mathfrak g)^{\mathfrak g}$ and $\mathcal B^W$.
\end{remark}
We now consider $\mathcal D=\hom_W(V,\mathcal B),$
and the following special element of $\mathcal D$
\begin{align}\label{f}
f_i(v)&=\pi(1\otimes \partial_v\psi_i),
\end{align}
where $\partial_v$ denotes the directional derivative in the direction $v\in V$ and $i=1,\ldots r$.

Notice that  with respect to a orthonormal basis $\{x_i\}$ of $V$, we have  $$f_i=\sum_{j=1}^r \pi(1\otimes \frac{\partial\psi_i}{\partial x_j})\otimes x_j.$$ 
Moreover, by \eqref{ruledelta}, for every $v\in V$,
\begin{equation}\label{ffff}
 \delta(f_i(v))=0.
\end{equation}

Fix a function $c:T\to \mathbb C$ constant on conjugacy classes as in the previous section.
Set  $|T|_c=c(s_\ell)|T|_\ell+c(s_p)|T|_p$ and define
\begin{equation}\label{u}u_i(v)=\frac {r}{2|T|_c }D_c f_i(v).\end{equation}
\begin{proposition}\label{secondo}  For every $v\in V$,
 $ \delta(u_i(v))=f_i(v)$.
\end{proposition}
\begin{proof} Apply Proposition \ref{ddelta} taking  $U=V$ (notice that we have assumed that $V$ is irreducible) and $x=f_i$.
Since $\chi_U(s_p)=\chi_U(s_\ell)=r-2$, the claim follows.
\end{proof}

From now on, we will  use the constant function $c=1$ on $T$ and set $D=D_1$.
Recall the  natural $\mathcal W$-valued bilinear form $E$ on $\mathcal W \otimes V$ defined by \eqref{formaE}
and its restriction to  a $\mathcal B^W$-valued bilinear form on the $\mathcal B^W$-module
$\mathcal D$.

\begin{proposition}\label{prop}\
\begin{enumerate}
\item $E(f_i,f_j)=0$.
\item Assume that $d_i\ne d_j$ for $i\ne  j$. Then 
\begin{equation}\label{eqm}
E(u_i,f_j)=E(u_j,f_i)=\begin{cases}Êk_{i,j}p_s \quad&\text{if there exists $s$ such that  $d_i+d_j-2=d_s$,}\\0\quad&\text{otherwise.}\end{cases}
\end{equation}
with $k_{i,j}\neq 0$.
Furthermore if $W$ is crystallographic, then 
\begin{equation}\label{k=c}k_{i,j}=c_{i,j},\end{equation} where  $c_{i,j}$ is the constant introduced in \eqref{cost}.
\end{enumerate}
\end{proposition}
\begin{proof} Let us choose a orthonormal basis $\{x_i\}$ for $V$. Then,
since $E(f_i,f_j)\in\mathcal B^W$ and 
$$E(f_i,f_j)=\sum_{s=1}^r\pi(1\otimes\frac{\partial \psi_i}{\partial x_s}\frac{\partial \psi_j}{\partial x_s}),$$
we have that  $ \sum_{s=1}^r(\partial \psi_i/\partial x_s)(\partial \psi_j/\partial x_s)\in J$, hence (1) follows.
 
 To see part (2), notice that 
  $E(u_i,f_j)\in (\bigwedge^1V\otimes \mathcal H)^W$ so that  if there is no $s$ for which $d_i+d_j-2=d_s$,  then by Proposition \ref{Solo} we have $E(u_i,f_j)=0$.
   
  Assume $d_i+d_j-2=d_s$. Then we have that necessarily $E(u_i,f_j)=k_{i,j} p_k,\,k_{i,j}\in\C$ again by Proposition \ref{Solo}.  
  
 We have to prove that $k_{i,j}\neq 0$. Lifting to  
$\mathcal W$ and applying $\delta$ we obtain 
$E(d\psi_i,d\psi_j)=k_{i,j} \psi_k+ b,\,b\in J^2$. If $k=r=2$ this statement is obvious. If $k=r$, so that the indices $i,j$ are complementary, we can then apply the argument of  Proposition 2.9 from \cite{DPP}  and deduce $k_{i,j}\ne 0$.

This completes the proof in the non crystallographic case, since  the only pairs $d_i,d_j$ with $d_i+d_j-2=d_s$  either have $d_i=2$ or $d_j=2$ or $i,j$ are complementary: this is clear for dihedral groups;  for $H_3$ the degrees are $2,6,10$ and for $H_4$ they are $2,12,20,30$ so everything is readily verified.

It remains to treat the crystallographic case. But this follows from \cite[2.7.2]{DPP}, from which also the equality $k_{i,j}=c_{i,j}$ is easily deduced. 
\end{proof}
\begin{remark} Type $D_{2n}$, where a basic degree of multiplicity $2$ appears, is handled as in \cite[Proposition 1.3]{DPP}.
\end{remark}
\begin{proposition}\label{PP2}  
\begin{equation}\label{mainn}E(u_j,u_i)
=0.\end{equation}
\end{proposition}
\begin{proof} Consider $u,v\in \mathcal B$.
We have
$$D(uv)=(Du)v+\sum_{s\in S}s(u)\nabla_s(v)=(Du)v-uDv+\sum_{s\in S}(u+s(u))\nabla_s(v).$$
But $(u+s(u))\nabla_s(v)=(-1)^{deg (u)}(1-s)(d\log\alpha_s(u+s(u))v)$
Since $d\log\alpha_s(u+s(u))$ is  fixed by $s$, we have
$$s((u+s(u))\nabla_s(v))=-(u+s(u))\nabla_s(v).$$ Thus, since the usual scalar product is $W$-invariant, we deduce that  $s((u+s(u))\nabla_s(v))$ is orthogonal to the $W$-invariants. So, also 
$$ (uv)-(Du)v+uDv$$ 
 is orthogonal to the $W$-invariants.
From this, reasoning as in \cite[Lemma 2.15]{DPP}, we get that
$$D\, E(f_j,u_i)- E((1\otimes D)f_j, u_i)+E(f_j,(1\otimes D)u_i)=0$$
However  by Lemma \ref{primo}, $D\, E(f_j,u_i)=0$, by Proposition \ref{secondo'} $E(f_j,(1\otimes D) u_i)=0$, so that \eqref{mainn} follows.
\end{proof}
\section{Main Theorem}
\begin{theorem}\label{main}
{\bf (1).}
  $\mathcal D$ is a free module, with basis the elements $f_i,u_i, i=1,\ldots,r,$ over the exterior algebra $\bigwedge(p_1,\ldots,p_{r-1}).$\par\noindent
{\bf (2).} Let $k_i=k_{i,r-i+1}$ with  $k_{i,j}$     defined as in \eqref{eqm}.
Then for each $i=1,\ldots,r,$ $$E(f_i,u_{r-i+1})=k_ip_r.$$   The multiplication by $p_r$ is self adjoint for the form $E$ and it is  given by the formulas
\begin{align}\label{perpr}
p_rf_i= -\sum_{j=1,\ j\neq i}^{r} k_j^{-1}E(f_i,u_{r-j+1})f_j,\qquad  i=1,\ldots,r,\\ \label{perpr2}
p_ru_i= -\sum_{j=1,\ j\neq i}^{r} k_j^{-1} E(f_i,u_{r-j+1})u_j ,\qquad i=1,\ldots,r,
\end{align} 
\end{theorem}
\begin{proof} {\bf (1).} Suppose that we have a relation $\sum\limits_{i=1}^r\lambda_iu_i+\sum\limits_{j=1}^r\mu_jf_j=0$. Then  apply $1\otimes \delta$ and by \eqref{ffff} and Proposition \ref{secondo} get $\sum\limits_{i=1}^r\lambda_if _i=0$.  So if we prove  that the $f_i$ are linearly independent, we get $\lambda_i=0$ for all $i$  and in turn that  also all the $\mu_j$ are $0$. 
\par
Remark that, if there is a non trivial relation $\sum\limits_{j=1}^r\mu_jf_j=0$,  we may assume that it  is homogeneous.  Moreover, given an index $j$,    multiplying by a suitable element of  $\bigwedge(p_1,\ldots,p_{r-1}) $ we can reduce ourselves to the case in which $\mu_j= p_1\wedge p_2\wedge \ldots \wedge p_{r-1} .$

Notice now  that the   coefficient $\mu_h$ of the terms $\mu_hf_h$ for which $d_h<d_j$  has  degree higher than the maximum allowed degree, hence it is zero. Thus, if we choose for $j$ the maximum for which $\mu_j\neq 0$, we are reduced to prove that  
\begin{equation}\label{ff}p_1\wedge p_2\wedge \ldots \wedge p_{r-1} f_j\neq 0.\end{equation} 
By part (2) of Proposition \ref{prop} we have  $E(f_j,u_{r-j+1})=k_r p_r$,  hence 
$$E(p_1\wedge p_2\wedge \ldots \wedge p_{r-1} f_j, u_{r-j+1})=k_r\, p_1\wedge p_2\wedge \ldots \wedge p_{r-1}\wedge p_r\ne 0.$$

{\bf (2).} Using Propositions \ref{prop}, \ref{PP2}, one can mimic the proof of \cite[Theorem 1.4]{DPP}. We briefly explain how to proceed, omitting for simplicity the case $D_{2n}$.

 
Consider the relation for $u_i$. We have
\begin{equation}
\label{conu}p_r u_i=\sum_{j=1}^r H_j u_j+\sum_{j=1}^r K_j f_j 
\end{equation}
where the $H_j,K_j\in \bigwedge(p_1,\ldots,p_{r-1}).$ Applying the differential $1\otimes\delta$ we get
\begin{equation}\label{p}p_r f_i=\sum_{j=1}^r H_j f_j.\end{equation}
Thus the relation for  $f_i$ involves only the $f_j$'s. Also we have that the relation is homogeneous.

For each $j$, taking the scalar product with $u_{r-j+1}$, we have 
\begin{align*} p_r E(f_i,u_{r-j+1})&=  H_j E(f_j,u_{r-j+1})+\sum_{h\neq j}H_h E(f_h,u_{r-j+1})\\
&=  H_j k_j p_r +\sum_{h\neq j}H_h E(f_h,u_{r-j+	1}).\end{align*}
Since the terms $\sum_{h\neq j}H_h E(f_h,u_{r-j+1})  $  do not involve $p_r$, we must have
\begin{align}\notag&\sum_{h\neq j}H_h E(f_h,u_{r-j+1}) =0,\\
&\label{1}- E(f_i,u_{r-j+1}) p_r=  H_j  k_j   p_r  \end{align}
If $i\neq j$  we have that $E(f_i,u_{r-j+1})$ is not a multiple of $p_r$ and we deduce that  
$$ E(f_i,u_{r-j+1}) = -k_j  H_j.$$
If $i= j$  we   deduce $H_j=0$, so finally \eqref{p}Ê becomes 
\begin{equation}
\label{lamol}p_r f_i+ \sum_{i\neq j} k_j^{-1}E(f_i,u_{r-j+1}) f_j=0.
\end{equation}
Since $E(f_i,u_{r-i+1})=k_ip_i$, formula \eqref{lamol} is indeed formula \eqref{perpr}, as required. We go back to formula \eqref{conu}, which we now write:
\begin{equation}\label{p3}p_r u_i= -\sum_{j=1}^r k_j^{-1}E(f_i,u_{r-j+1}) u_j+\sum_{j=1}^r K_j f_j.\end{equation}
Take the the scalar product of both sides of \eqref{p3} with $u_{r-j+1}$. We get
$$p_r E(u_i,u_{r-j+1})= -\sum_{j=1}^r k_j^{-1}E(f_i,u_{r-j+1}) E(u_j,u_{r-j+1})+\sum_{j=1}^r K_j E(f_j,u_{r-j+1}).$$
Since $E(u_h,u_k)=0$, we deduce that 
\begin{equation*}k_jK_j   p_r +\sum_{i,\,i\neq j} K_i E( f_i ,u_{r-j+1})=0.
\end{equation*}
We claim that all $K_j$ are zero. Indeed    the only product containing $p_r$ is $k_jK_j   p_r $. Since  each element of $\Gamma$ can be written in a unique way in the form $a+b p_r$ with $a,b\in\bigwedge(p_1,\ldots,p_{r-1})$, we deduce $K_j=0$ as desired.
\par

\end{proof}
Using \eqref{k=c} one gets the following Corollary, which obviously implies Reeder's conjecture \eqref{RCC} for $\g$.
\begin{corollary}\label{CC}The map 
$$p_i\mapsto P_i,\quad u_i\mapsto u^\wedge_i,\quad f_i\mapsto f^\wedge_i,\quad 1\leq i\leq r,$$
extends to  an isomorphism of graded $\mathcal B^W$-modules 
$( \bigwedge\h\otimes \mathcal H\otimes\h)^W\to (\bigwedge \g\otimes \g)^\g.$
\end{corollary}
\vskip5pt
\section{The Weyl group side of the  little adjoint representation}\label{la}
Suppose that $W$ contains two distinct conjugacy classes of reflections $T_\ell, {T_p}$. Set  $r_\ell=|T_\ell\cap S|$, $r_p=|T_p\cap S|$. 
Denote by  $H_{T_\ell}$  the subgroup of $W$ generated by the reflections $s\in {T_\ell}$, 
and by $W_{{T_p}}$ the reflection subgroup of $W$ generated by the reflections $s\in {T_p}\cap S$. 
The following fact is proven in \cite[Proposition 2.1]{Pan}. 
\begin{lemma} 
 $W= W_{{T_p}}\ltimes H_{T_\ell}$ so $W/H_{T_\ell}$ is canonically isomorphic to $W_{{T_p}}.$
 Symmetrically 
 $W= W_{{T_\ell}}\ltimes H_{T_p}$, so $W/H_{T_p}$ is canonically isomorphic to $W_{{T_\ell}}.$
\end{lemma}
Let us now consider the reflection representation $U$ of $W_{T_p}$. Since $W_{T_p}$ is a quotient of $W$, we may consider $U$  as a $W$-module.

Consider now  $V$ as a ${H_{T_\ell}}$-module. Since ${H_{T_\ell}}$ is generated by reflections, the ring $A^{H_{T_\ell}}$ is a polynomial ring generated by homogeneous generators $\overline \psi_1, \ldots \overline \psi_n$. Let  $J_{H_{T_\ell}}$ be the ideal in $A^{H_{T_\ell}}$ generated by $\overline \psi_1, \ldots \overline \psi_n$. Clearly $W$ acts on  $\overline V=J_{H_{T_\ell}}/J_{H_{T_\ell}}^2$, and we have
\begin{proposition}\label{piccola} For the $W$-module $\overline V$ one has \begin{enumerate}\item $\overline V\simeq U\oplus \overline V^W.$ 
 \item dim\ $\overline V^W=|{T_\ell}\cap S|.$
 \item The submodule $U\subset \overline V$  is homogeneous of degree $d_n/2-(r_p-1)r_\ell$. \end{enumerate}
\end{proposition}
\begin{proof} 
The proof is a case by case check. Let us start recalling that we have two distinct conjugacy classes of reflections precisely in the following cases: $B_n=C_n$, $I_2(2m)$, $F_4$. 
\vskip5pt
{\sl Type $B_n$.}   Let us choose an orthonormal basis $e_1,\ldots, e_n$ of $V$ in such a way that  the  conjugacy class ${T_\ell}$ is  given  by the $n$ reflections with respect to the coordinate hyperplanes, the other class  ${T_p}$ by the reflections with respect to the hyperplanes of equation $x_i\pm x_j$, $i<j$.

The group $H_{T_\ell}$ is clearly isomorphic to $(\mathbb Z/2\mathbb Z)^n$, and identifying $A$ with $K[x_1,\ldots ,x_n]$ using the coordinates associated to our basis, it turns out that   $A^{H_{T_\ell}}=K[x_1^2,\ldots ,x_n^2]$. Moreover, $W_{T_p}$ is the symmetric group $S_n$, acting  on 
$\overline V=\langle x_1^2,\ldots ,x_n^2\rangle$ by the permutation representation. It clearly follows that $\overline V=U\oplus \overline V^W$. So $\overline V$ and hence $U$ is contained in the homogeneous component of degree 2 and the rest is clear since $d_n=2n$, $r_\ell=1$
$r_p=n-1$, so that $q=d_n/2-(r_p-1)r_\ell=2$.

Let us now exchange the roles of ${T_\ell}$ and ${T_p}$. In this case $H_{T_p}$ is a Weyl group of type $D_n$, so $H_{T_p}$ has index 2 in $W$ and $W_{T_\ell}\simeq \mathbb Z/2\mathbb Z$. We have that 
$$A^{H_{T_p}}=K[\psi_0,\psi_1,\cdots,\psi_{n-1}],$$ where $\psi_i=\sum_{h=1}^{n}x^{2i}$ is a basic invariant for $B_n$ of degree $2i$ for $i=1,2,\cdots,n-1$, while $\psi_0=x_1\cdots x_n$. It is now clear that $\overline V=\langle \psi_0,\ldots , \psi_1,\ldots \psi_n\rangle$ and $U=K\psi_0,$ while $\langle \psi_1,\ldots \psi_n\rangle=\overline V^W$.  The remaining statement  is clear. \vskip5pt
{\sl Type $I_2(2m)$.} In this case the roles of ${T_\ell}$ and ${T_p}$  are completely symmetric, so we shall treat only one case. 
We have 
$$H_{T_\ell}=I_2(m), \quad W_{T_p}\simeq \mathbb Z/2\mathbb Z. \quad A^{H_{T_\ell}}=K[\psi_1,\psi_2],$$ while $A^{W}=K[\psi_1,\psi^2_2]$ with deg $\psi_1=2$, deg $\psi_2=m$. From this everything follows.
\vskip5pt
{\sl  Type $F_4$.} Also in this case the roles of ${T_\ell}$ and ${T_p}$  are completely symmetric, so we shall treat only one case. The group
 $H_{T_\ell}$ is of type $D_4$ and $W_{T_p}=S_3$. Let $\psi_1,\psi_2,\psi_3, \psi_4$ be basic invariants for $H_{T_\ell}$ of degrees $2,4,4,6$ respectively. The basic invariants for $W$ occur in degrees $2,6,8,12$. We can choose $\psi_1,\psi_4$ to be basic invariants for $W$. We claim that the action of $W_{T_p}$ on $\langle \psi_2,\psi_3\rangle$ is given by its reflection representation. Indeed, since $\langle \psi_2,\psi_3\rangle$ cannot contain invariants for $W_{T_p}$, the only other possibility is that $W_{T_p}$ acts on  $\langle \psi_2,\psi_3\rangle$ by two copies of the sign representation. If this were the case we would have that the degree 8 component of $ A^W$ would have dimension at least  $5$ while  we know that it has dimension $3$. Finally $d_n=12$, $r_\ell=r_2=2$ so $d_n/2-(r_p-1)r_\ell=4$.
\end{proof}

\vskip5pt

Now take a $W$-invariant complement to $J_{H_{T_\ell}}^2$ in $J_{H_{T_\ell}}$ which we can clearly identify with  $\overline V$. Then $A^{H_{T_\ell}}=K[\overline V]=K[U]\otimes K[\overline V^W]$.    Set  $\tilde A=K[U].$ Let $\phi_1,\ldots \phi_{r_p}$ be  homogeneous polynomial generators of $\tilde A^{W}$. Consider the ideal $J$ kernel  of the quotient $\pi:A\to \mathcal H$.
Then
\begin{lemma}\label{invario} $\tilde J:=J\cap \tilde A$ is the ideal generated by $\phi_1,\ldots \phi_{r_p}$ .\end{lemma}
\begin{proof} Take a homogeneous basis $\phi_{r_p+1},\ldots \phi_n$ of $\overline V^W$. Then $J=(\phi_1,\ldots  ,\phi_n)$. Take $a\in J\cap A^{H_{T_\ell}}$, and write
$a=\sum_i b_i\phi_i$, with $b_i\in A$. Applying to $a$  the operator 
$R=\frac{1}{|H_{T_\ell}|}\sum_{g\in H_{T_\ell}}g$ we get
$$a=\sum_iR(b_i)\phi_i,$$ so that, since $R(b_i)\in A^{H_{T_\ell}}$, 
$J\cap A^{H_{T_\ell}}$ is generated by $\phi_1,\ldots \phi_n$. But then $\tilde J$ is clearly generated by $\phi_1,\ldots \phi_{r_p}$. 
\end{proof}

 Let us  now double all degrees. The inclusion $\tilde A\subset A$ multiplies the degrees by $q=d_n-2(r_p-1)r_\ell$. Furthermore  Lemma \ref{invario} clearly implies that we have an inclusion of $\tilde{\mathcal H}=\tilde A/\tilde J$ into 
$\mathcal H$, which also multiplies the degrees by $q=d_n-2(r_p-1)r_\ell$.

 In each case $W_{T_p}$ is the symmetric group $S_{r_p+1}$, so that deg $\phi_i=2(i+1)q$, each $j=1,\dots ,r_p$.  In particular $\phi_{r_p}$ has degree $(d_n-2(r_p-1)r_\ell)(r_p+1)$, which one checks easily to equal  $2d_n$. We deduce that $\phi_{r_p}$ is a highest degree generator of  both $\tilde A^W$ and $A^W$.

We  define for each $i=1,\ldots,r_p$,  the $W$-equivariant map
$g_i:U\to \bigwedge V\otimes \mathcal H,$
given, for $u\in U$, by $$g_i(u)=1\otimes\pi' (\partial_u\phi_i)=\sum_{j=1}^s(u,y_j)(1\otimes \pi'(\frac{\partial \phi_i}{\partial y_j})).$$
By the above discussion $g_i$ is homogeneous of degree $2iq$. 
Let us now take the operator $D$ introduced in \eqref{partial} Ê(with $c=1$) and notice that clearly its restriction to $A^{H_{T_\ell}}$ equals 
$$D_{(p)}:=\sum_{s\in T_p}\nabla_s.$$ 
Define
\begin{equation}v_i(u)=\frac {s}{2|T_p|}D_{(p)} g_i(v).\end{equation}
By repeating the proof of Proposition \ref{secondo}, we then get
 $(1\otimes \delta)(v_i)=g_i$.
Furthermore  using  a $W$-invariant bilinear form on $U$ and reasoning exactly as in \ref{formaE}, we obtain a bilinear for on the module $\mathcal E=\hom_W(U,\mathcal B)$ with values in $\mathcal B^W$ which we still denote by $E$. We have

\begin{proposition}\label{prop23} Let $1\leq i,j \leq r_p$. Then
\begin{enumerate}
\item   $E(g_i,g_j)=0$.
\item
\begin{equation}
E(v_i,g_j)=E(v_j,g_i)=\begin{cases}Êm_{i,j}p_k \quad&\text{if there exists $k$ such that  $d_i+d_j-2=d_k$,}\\0\quad&\text{otherwise.}\end{cases}
\end{equation}
with $m_{i,j}\neq 0$.
\end{enumerate}
\end{proposition}
\begin{proof} 
We have seen that any set $\phi_1,\ldots \phi_{r_p}$ of  homogeneous polynomial generators of $\tilde A^{W}$ is part of a set of polynomial generators for $A^{W}$ and that $\phi_{r_p}$ is the highest degree generator for  both  $\tilde A^{W}$ and $ A^{W}.$ Furthermore,  by Proposition \ref{ddelta}, we have that $(1\otimes \delta)(v_i)=g_i$  for all $i=1,\dots ,r_p$.
At this point,  everything follows right away from Propositions \ref{prop23} and \ref{PP2} applied to the group $W_{T_p}$.
\end{proof}
Let us now  now consider the $\mathcal B^W$-module $\mathcal D_p:=\hom_W(U,\mathcal B)$. We get, repeating word by word the proof of Theorem \ref{main},
\begin{theorem}\label{main45}
{\bf (1).}
  $\mathcal D_p$ is a free module, with basis the elements $g_i,v_i, i=1,\ldots,r_p,$ over the exterior algebra $\bigwedge(p_1,\ldots,p_{r_p-1}).$\par\noindent
{\bf (2).} The multiplication by $p_{r_p}$ is self adjoint for the form $E$. Setting $m_i=m_{i,r+1-i}$,  it is  given by the formulas
\begin{align}
p_r g_i&= -\sum_{j=1,\ j\neq i}^{r_p} m_j^{-1}E(g_i,v_{r_p-j+1})g_j,\qquad  i=1,\ldots,r_p,\\
p_r v_i&= -\sum_{j=1,\ j\neq i}^{r_p} m_j^{-1} E(g_i,v_{r_p-j+1})v_j ,\qquad i=1,\ldots,r_p.
\end{align} 
\end{theorem}
In the case in which $W$ is the Weyl group of a simple Lie algebra $\g$, which is of course non simply laced,  our representation  $U$ is the zero weight space of the irreducible $\g$-module $\g_s$ whose highest weight is the dominant short root, which in fact is small. Using Theorem \ref{main45}, and \cite{DPP2}, one can then  easily deduce the following Corollary which obviously implies Reeder's conjecture \eqref{RCC} for $\g_s$.

\begin{corollary}\label{C}The map 
$$p_i\mapsto P_i,\quad v_i\mapsto u^\wedge_i,\quad g_i\mapsto f^\wedge_i,\quad 1\leq i\leq r$$
extends to  an isomorphism of graded $\mathcal B^W$-modules 
$(\bigwedge\h\otimes \mathcal H\otimes U)^W\to (\bigwedge \g\otimes \g_s)^\g.$
\end{corollary}

\section{A possible extension of Reeder's conjecture}\label{FR}
  Consider the bracket map $[-,-]:\bigwedge^2\g\to \g$.
Dualizing and using the isomorphism $\g\simeq \g^*$ given by the Killing form, we get a linear map
$\g\to\bigwedge^2\g$. Since $\bigwedge^{even}\g$ is a commutative algebra, this linear map extends to  homomorphism of algebras   $s:S(\g)\to \bigwedge^{even}\g$.

The inclusion $\h\subset \g$ also gives an inclusion of rings $j:S(\h)\to S(\g).$
Composing with $s$ we get the homomorphism,   $\tau:S(\h)\to \bigwedge^{even}\g$.

Let, as in Section 1, $J$  be  the ideal in $S(\h)$ generated by the $W-$invariants vanishing in $0$.  Recall that the ideal $J$ has a canonical complement $\mathcal A$,  the so called harmonic polynomials, i.e. those elements in $S(\h)$ killed by all constant coefficients $W$-invariant differential operators  without constant term. 

We have
\begin{proposition} The restriction of the homomorphism $  \tau :S(\h)\to \bigwedge^{even}\g$ to $\mathcal A$ is injective.
\end{proposition}
\begin{proof} Let $\Dp\subset \h^*\simeq \h$ denote the set of positive roots. Take the Weyl denominator polynomial $P=\prod_{\alpha\in \Dp}\alpha= \prod_{\a\in\Dp} t_\a$ (where $t_\a\in \h$ is defined by $\lambda(t_\a)=(\a,\lambda),\,\lambda\in \h^*$). We know that $W$ acts on $P$ by the sign representation and that in degree $N=|\Dp|$ the homogeneous  component $\mathcal A_N$ of $\mathcal A$ is spanned by $P$. 

Recall from \cite[(89)]{K} that, if $\{x_i\}, \{x^i\}$ are dual basis of $\g$ w.r.t. the chosen invariant form, then  $s(x)=1/2\sum_i x_i\wedge [x^i,x],\,x\in \g$. Fix now root vectors $e_\beta,\,\beta\in \Dp$ and 
choose $e_{-\beta}$ such that $(e_\beta, e_{-\beta})=1$.  
A simple computation using the above formula for $s$ shows that
for any $\alpha\in \Dp$we have 
$$\tau (t_\alpha)=\sum_{\beta\in \Dp}(\beta,\alpha)e_{\beta}\wedge e_{-\beta}.$$
It follows that 
$$\tau (P)=per(A)\prod_{\beta\in \Dp}(e_{\beta}\wedge e_{-\beta}),$$
where $per(A)$ is the permanent
of the matrix
$A=((\beta,\alpha)).$\par
Now $A$ is a positive semidefinite matrix. It follows that its permanent is non zero. Indeed by  \cite{MM}, one has
$$per(A)\geq \frac{N!}{(\rho,\rho)^N}P(\rho)^2=\frac{N!}{(\rho,\rho)^N}\prod_{\a\in\Dp}(\a,\rho)^2>0,$$ where $\rho$ is the half sum of positive roots, which is well-known to be regular. This proves our claim in degree $N$.

Let us now consider $\mathcal A_m$. We have $m\leq N$ otherwise $\mathcal A_m=\{0\}$ and there is nothing to prove. So we can assume $m<N$.  Take $0\ne a\in \mathcal A_m$. We then know that there is an element $b\in \mathcal A_{N-m}$ such that
$ab=P+r$ with $r\in J_N.$ Assume $\tau(a)=0$. Then $\tau(r)=-\tau(P)$. Consider the $W-$module  $U$ spanned by $r$. Then $U\subset J_N$ and $\tau$ gives a surjective $W-$equivariant homomorphism $U\to \mathbb C\tau(P)$. We deduce that $U$ and hence $J_N$ contains a copy of the sign representation of $W$, contrary to the fact that $P$ spans the only copy of the sign representation in degree $N$. It follows that $\tau(a)\neq 0$, proving our claim.
\end{proof}
Recall that there is a  $W$-equivariant degree preserving isomorphism between $\mathcal A^*$ and $\mathcal H$. 
Since $\wedge^{even}\g$ is selfdual, dualizing $\tau$ we obtain a surjective degree preserving  map
$$\phi:\bigwedge^{even}\g\to \mathcal H.$$


Let $p$ be the projection to $p:\bigwedge \g\to \bigwedge\h$ and $\pi:\bigwedge \g\to \bigwedge^{even}\g$  the projection on the even part. Using these, we can build up the  map 
\begin{equation}\label{map}  \Phi:\bigwedge \g \xrightarrow {\wedge^*} {\bigwedge}  \g \otimes {\bigwedge}\g\xrightarrow {Id\otimes \pi}     \bigwedge \g 	\otimes \bigwedge^{even}\g\xrightarrow {p\otimes \phi } \bigwedge \h\otimes \mathcal H  .
\end{equation}
 Let $V$ be any finite dimensional irreducible $\g$-module. Denote by $V^0$ its zero weight space and by $i:V^0\hookrightarrow V$  the natural inclusion. If $f\in \hom(V,\bigwedge \g)$, we may consider $$\Phi^V_f:=\Phi\circ f\circ i\in \hom(V_0,\bigwedge\h\otimes \mathcal H).$$ Clearly, by equivariance,
 if $f\in \hom_\g(V,\bigwedge \g)$ then 
$\Phi^V_f\in \hom_W(V^0,\bigwedge\h\otimes \mathcal H)$.

Hence we have a graded map 
\begin{equation}\label{Phiii}
\Phi^V:\hom_\g(V,\bigwedge \g)\to \hom_W(V^0,\bigwedge\h\otimes \mathcal H),\qquad \Phi^V(f)=\Phi^V_f.
\end{equation}
\vskip5pt
\noindent{\bf Conjecture.}  For any finite dimensional  irreducible $\g$-module $V$, the map  $\Phi^V$ is injective.
\begin{remark} Since $\dim hom_\g(V,\bigwedge \g)=\dim \hom_W(V^0,\bigwedge\h\otimes \mathcal H)$ if (and only if) $V$ is small (cf. \cite[Corollary 4.2]{R}), the above conjecture implies Reeder's conjecture.
\end{remark}


\end{document}